\documentclass{svjour3}
\smartqed  
\usepackage{graphicx,amsmath,amsfonts,amssymb,enumerate,mathrsfs,mathtools}

\newcommand{\N}{\mathbb{N}}                   
\newcommand{\R}{\mathbb{R}}                   
\newcommand{\A}{\mathscr{A}_a}                
\newcommand{\AR}{\mathscr{A{\!}R}}            
\newcommand{\Zar}{\mathrm{Zar}}
\newcommand{\Na}{\mathcal{N}}                 
\newcommand{\Reg}{\mathrm{Reg}}
\newcommand{\Sing}{\mathrm{Sing}}
\newcommand{\I}{\mathcal{I}}
\newcommand{\V}{\mathcal{V}}
\newcommand{\rad}{\mathrm{rad}}

\journalname{Math. Ann.}

\begin{document}

\title{A proof of Kurdyka's conjecture on arc-analytic functions
 \thanks{J. Adamus's research was partially supported by the Natural Sciences and Engineering Research Council of Canada.}}
\titlerunning{Kurdyka's conjecture on arc-analytic functions}
\author{Janusz Adamus \and Hadi Seyedinejad}
\institute{J. Adamus \at Department of Mathematics, University of Western Ontario, London, ON, Canada N6A 5B7\\
 \email{jadamus@uwo.ca}
  \and H. Seyedinejad \at Department of Mathematics, University of Western Ontario, London, ON, Canada N6A 5B7\\
	 \email{sseyedin@uwo.ca}}

\date{Received: / Accepted: }
\maketitle

\begin{abstract}
We prove a conjecture of Kurdyka stating that every arc-symmetric semialgebraic set is precisely the zero locus of an arc-analytic semialgebraic function. This implies, in particular, that arc-symmetric semialgebraic sets are in one-to-one correspondence with radical ideals of the ring of arc-analytic semialgebraic functions.
\end{abstract}


\section{Introduction}
\label{sec:intro}

A set $X$ in $\R^n$ is called \emph{semialgebraic} if it can be written as a finite union of sets of the form $\{x\in\R^n:p(x)=0,q_1(x) >0,\dots,q_r(x)>0\}$, where $r\in\N$ and $p,q_1,\dots,q_r\in\R[x_1,\dots,x_n]$. Given $X\subset\R^n$, a \emph{semialgebraic function} $f:X\to\R$ is one whose graph is a semialgebraic subset of $\R^{n+1}$. A function $f:X\to\R$ is called \emph{arc-analytic} if it is analytic along every arc, that is, if $f\circ\gamma$ is analytic for every real-analytic $\gamma:(-1,1) \to X$. 

Let $X$ be a semialgebraic subset of $\R^n$. We say that $X$ is \emph{arc-symmetric} if, for every analytic arc $\gamma:(-1,1)\to\R^n$ with $\gamma((-1,0))\subset X$, we have $\gamma((-1,1))\subset X$. Our main result is the following affirmative answer to a conjecture of Kurdyka {\cite[Conj.\,6.3]{Ku}}:

\begin{theorem}
\label{thm:main}
Let $X$ be an arc-symmetric semialgebraic set in $\R^n$. There exists an arc-analytic semialgebraic function $f:\R^n\to\R$ such that $X=f^{-1}(0)$. 
\end{theorem}

The elegant theory of arc-symmetric semialgebraic sets was developed by Kurdyka \cite{Ku}. It is based on a fundamental observation (\cite[Thm.\,1.4]{Ku}) that the arc-symmetric semialgebraic sets are precisely the closed sets of a certain noetherian topology on $\R^n$. (A topology is called \emph{noetherian} when every descending sequence of its closed sets is stationary.) Following \cite{Ku}, we will call it the \emph{$\AR$ topology}, and the arc-symmetric semialgebraic sets will henceforth be called \emph{$\AR$-closed sets}.

Noetherianity of the $\AR$ topology allows one to make sense of the notions of irreducibility and components of a semialgebraic set much like in the algebraic case (see Section~\ref{sec:prelim} for details). The class of $\AR$-closed sets includes, in particular, the algebraic sets as well as the Nash analytic sets (in the sense of \cite{BCR}). The $\AR$ topology is strictly finer than the Zariski topology on $\R^n$ (see, e.g., \cite[Ex.\,1.2]{Ku}). Moreover, it follows from the semialgebraic Curve Selection Lemma that $\AR$-closed sets are closed in the Euclidean topology in $\R^n$ (see \cite[Rem.\,1.3]{Ku}).

\begin{remark}
\label{rem:comp}
It is interesting to compare the $\AR$ topology with other noetherian refinements of Zariski topology, in which closed sets are defined as the zero loci of functions more general than polynomials, such as Nash functions or continuous rational functions. Continuous rational (a.k.a. \!\emph{regulous} over $\R^n$) functions have recently attracted attention of numerous authors (see, e.g., \cite{FHMM}, \cite{KK}, \cite{KN}, and \cite{Kuch}).  The refinement arising from Nash functions was studied in \cite{FG}. It turns out that the $\AR$ topology is strictly finer than both Nash and regulous topologies, which incidentally are incomparable. This topic is studied in detail in a subsequent paper by the second author.
\end{remark}

Given an $\AR$-closed set $X$ in $\R^n$, we denote by $\A(X)$ the ring of arc-analytic semialgebraic functions on $X$. The elements of $\A(X)$ play the role of `regular functions' in $\AR$ geometry. Indeed, it is not difficult to see (\cite[Thm.\,5.1]{Ku}) that the graph as well as the zero locus of every arc-analytic semialgebraic mapping $f:X\to\R^m$ are $\AR$-closed as well. On the other hand, up until now it was not known whether every $\AR$-closed set may be realized as the zero locus of an arc-analytic function. Our Theorem~\ref{thm:main} fills this gap in the theory. It shows that $\AR$ topology is in fact the one defined by arc-analytic semialgebraic functions, which is not at all apparent from its definition.
\medskip

In the next section, we recall basic notions and tools used in this article. Theorem~\ref{thm:main} is proved in Section~\ref{sec:proof}. 
In Section~\ref{sec:duality}, we prove an $\AR$ version of the duality between closed sets and radical ideals, which is an easy consequence of Theorem~\ref{thm:main}. The last section contains some initial results concerning extension of arc-analytic functions.
We conjecture that $\A(X)\simeq\A(\R^n)/\I(X)$ for every $\AR$-closed set $X$ in $\R^n$.


\section{Preliminaries}
\label{sec:prelim}

First, we shall recall several properties of $\AR$-closed sets that will be used throughout the paper. For details and proofs we refer the reader to \cite{Ku}.

An $\AR$-closed set $X$ is called \emph{$\AR$-irreducible} if it cannot be written as a union of two proper $\AR$-closed subsets. It follows from noetherianity of the $\AR$ topology (\cite[Prop.\,2.2]{Ku}) that every $\AR$-closed set admits a unique decomposition $X=X_1\cup\dots\cup X_r$ into $\AR$-irreducible sets satisfying $X_i\not\subset\bigcup_{j\neq i}X_j$ for each $i=1,\dots,r$. The sets $X_1,\dots,X_r$ are called the \emph{$\AR$-components} of $X$.

For a semialgebraic set $E$ in $\R^n$, let $\overline{E}^{\Zar}$ denote the Zariski closure of $E$, that is, the smallest real-algebraic subset of $\R^n$ containing $E$. Similarly, let $\overline{E}^{\AR}$ denote the $\AR$-closure of $E$ in $\R^n$. Consider the following three kinds of dimension of $E$:
\begin{itemize}
\item the geometric dimension $\dim_{\mathrm{g}}\!E$, defined as the maximum dimension of a real-analytic submanifold of (an open subset of) $\R^n$ contained in $E$,
\item the algebraic dimension $\dim_a\!E$, defined as $\dim\overline{E}^{\Zar}$,
\item the $\AR$ topological (or Krull) dimension $\dim_\mathrm{K}\!E$, defined as the maximum length $l$ of a chain $X_0\subsetneq X_1\subsetneq\dots\subsetneq X_l\subset\overline{E}^{\AR}$, where $X_0,\dots,X_l$ are $\AR$-irreducible.
\end{itemize}
It is well known that $\dim_{\mathrm{g}}\!E=\dim_a\!E$ (see, e.g., \cite[Sec.\,2.8]{BCR}). By \cite[Prop.\,2.11]{Ku}, we also have $\dim_a\!E=\dim_\mathrm{K}\!E$.
We shall denote this common dimension simply as $\dim{E}$. By convention, $\dim\varnothing=-1$.
\medskip

An essential tool in our proofs is the blowing-up of $\R^n$ at a Nash subset. Recall that a subset $Z$ of a semialgebraic open $U\subset\R^n$ is called \emph{Nash} if it is the zero locus of a Nash function $f:U \to \R$. A function $f:U\to\R$ is called a \emph{Nash function} if it is an analytic algebraic function on $U$, that is, a real-analytic function such that there exists a non-zero polynomial $P\in\R[x,t]$ with $P(x,f(x))=0$, for every $x \in U$. We denote the ring of all Nash functions on $U$  by $\Na(U)$. We refer the reader to \cite[Ch.\,8]{BCR} for details on Nash sets and mappings.

Let $Z$ be a Nash subset of $\R^n$. Consider the ideal $\I(Z)$ in $\Na(\R^n)$ of all Nash functions on $\R^n$ vanishing on $Z$. By noetherianity of $\Na(\R^n)$ (see, e.g., \cite[Thm.\,8.7.18]{BCR}), there are $f_1,\dots,f_r \in \Na(\R^n)$ such that $\I(Z)=(f_1,\dots,f_r)$. Set
\[
\widetilde{R} \coloneqq \{(x,[u_1,\dots,u_r])\in\R^n\times\mathbb{P}^{r-1} : \ u_if_j(x)=u_jf_i(x) \mathrm{\ for\ all\ }i,j=1,\dots,r\}\,.
\]
The restriction $\sigma:\widetilde{R}\to\R^n$ to $\widetilde{R}$ of the canonical projection $\R^n\times\mathbb{P}^{r-1}\to\R^n$ is the \emph{blowing-up of $\R^n$ at (the centre) $Z$}. One can verify that $\widetilde{R}$ is independent of the choice of generators $f_1,\dots,f_r$ of $\I(Z)$.
Since a real projective space is an affine algebraic set (see, e.g., \cite[Thm.\,3.4.4]{BCR}), one can  assume that $\widetilde{R}$ is a Nash subset of $\R^N$ for some $N\in\N$. If $X$ is a Nash subset of $\R^n$, then the smallest Nash subset $\widetilde{X}$ of $\widetilde{R}$ containing $\sigma^{-1}(X\setminus Z)$ is called the \emph{strict transform of $X$ (by $\sigma$)}. In this case, if $Z\subset X$, then we may also call $\widetilde{X}$ the \emph{blowing-up of $X$ at $Z$}.
\medskip

For a semialgebraic set $E$ and a natural number $d$, we denote by $\Reg_d(E)$ the locus of those points $x\in E$ at which $E_x$ is a germ of a $d$-dimensional analytic manifold. If $\dim{E}=k$, we set $\Sing(E)\coloneqq E\setminus\Reg_k(E)$. Then, $\Sing(E)$ is semialgebraic and $\dim\Sing(E)<\dim{E}$.

Finally, recall that every algebraic set $X$ in $\R^n$ admits an \emph{embedded desingularization}. That is, there exists a proper mapping $\pi:\widetilde{R}\to\R^n$ which is the composition of a finite sequence of blowings-up with smooth algebraic centres, such that $\pi$ is an isomorphism outside the preimage of the singular locus $\Sing(X)$ of $X$, the strict transform $\widetilde{X}$ of $X$ is smooth, and $\widetilde{X}$ and $\pi^{-1}(\Sing(X))$ simultaneously have only normal crossings. (The latter means that every point of $\widetilde{R}$ admits a (local analytic) coordinate neighbourhood in which $\widetilde{X}$ is a coordinate subspace and each hypersurface $H$ of $\pi^{-1}(\Sing(X))$ is a coordinate hypersurface.)
For details on resolution of singularities we refer the reader to \cite{BM2} or \cite{Hi}.


\section{Proof of the main theorem}
\label{sec:proof}

In \cite[Thm.\,6.2]{Ku}, Kurdyka showed that, given an $\AR$-closed set $X$ in $\R^n$, there exists $f\in\A(\R^n)$ such that $X\subset f^{-1}(0)$ and $\dim(f^{-1}(0)\setminus X)<\dim{X}$. Our proof of Theorem~\ref{thm:main} follows the general idea of the above in that we lift the problem to a desingularization of the Zariski closure of $X$. We later lift it once more by a single blowing-up, to control the excess of zeros of $f$.

\subsubsection*{Proof of Theorem~\ref{thm:main}}
Let $X$ be an $\AR$-closed set in $\R^n$. We argue by induction on dimension of $X$.

If $\dim X\leq0$, then $X$ is just a finite set and hence the zero locus of a polynomial function. Suppose then that $\dim X=k>0$, and every $\AR$-closed set of dimension smaller than $k$ is the zero locus of an arc-analytic semialgebraic function on $\R^n$.
We may also assume that $X$ is $\AR$-irreducible.

Let $Y$ be the singular locus of $\overline{X}^{\Zar}$. Let $\pi:\widetilde{R}\to\R^n$ be an embedded desingularization of $\overline{X}^{\Zar}$\!, and let $\widetilde{X}$ be the strict transform of $\overline{X}^{\Zar}$. Then, $\pi$ is an isomorphism outside the preimage of $Y$.
Since $X \cap Y$ is an $\AR$-closed set of dimension smaller than $k$, the inductive hypothesis implies that there exists $h \in\A(\R^n)$ such that $X \cap Y = h^{-1}(0)$. 

By \cite[Thm.\,2.6]{Ku}, there exists a connected component $E$ of $\widetilde{X}$, such that the Euclidean closure $\overline{\Reg_k(X)}$ is equal to $\pi(E)$. Further, let $D \coloneqq \pi^{-1}(Y)$ and $Z \coloneqq E \cap D$.  Let $\sigma:\widehat{R}\to\widetilde{R}$ be the blowing-up of $\widetilde{R}$ at $Z$. As discussed in Section~\ref{sec:prelim}, we can assume that $\widehat{R}\subset\R^N$ for some $N\in\N$. Let $\widehat{E}$ and $\widehat{D}$ be the strict transforms of $E$ and $D$ by $\sigma$, respectively. Since $E$ and $D$ have only normal crossings, $\widehat{E}$ and $\widehat{D}$ are disjoint subsets of $\widehat{R}$. 

By the semialgebraic Tietze-Urysohn Theorem (\cite[Prop.\,2.6.9]{BCR}), disjoint closed semialgebraic subsets can be separated by open semialgebraic sets. Let then $U_1$ and $U_2$ be open semialgebraic subsets of $\R^N$ such that $\widehat{E}\subset U_1$, $\widehat{D}\subset U_2$, and $U_1\cap U_2=\varnothing$. Define a Nash function $q:U_1 \cup U_2 \to \R$ as
\[
q(z)\coloneqq\begin{cases} 0, &z\in U_1\\ 1, &z \in U_2\,.\end{cases}
\]
By the Efroymson extension theorem (see \cite{E} or \cite[Thm.\,8.9.12]{BCR}), the function $q$ admits a Nash extension to the whole $\R^N$; that is, there exists $g \in \Na(\R^N)$ such that $g|_{\widehat{E}\cup\widehat{D}}=q|_{\widehat{E}\cup\widehat{D}}$. Moreover, the set $\widehat{E}\cup\widehat{D}$ being closed Nash in $\R^N$, there exists $v \in \Na(\R^N)$ such that $\widehat{E}\cup\widehat{D}=v^{-1}(0)$. Now, define
\[
\widehat{f} \coloneqq \big(g \cdot (h \circ \pi \circ \sigma)\big)^2 + v^2\,.
\]
Observe that $\widehat{f}$ is an arc-analytic function defined on $\R^N$ (hence, in particular, on $\widehat{R}$), $\widehat{f}=(h \circ \pi \circ \sigma)^2$ on $\widehat{D}$, $\widehat{f}=0$ on $\widehat{E}$, and $\widehat{f}$ never vanishes outside of $\widehat{E}\cup\widehat{D}$.

Next, we push down $\widehat{f}$ by $\sigma$ in order to get an arc-analytic function on $\widetilde{R}$. More precisely, we define $\widetilde{f}:\widetilde{R}\to\R$ as
\[
\widetilde{f}(y)\coloneqq\begin{cases}((\widehat{f}\circ\sigma^{-1})\cdot(h\circ\pi))(y), &y\notin Z\\ 0, &y\in Z\,.\end{cases}
\]
To see that $\widetilde{f}$ is arc-analytic, let $\widetilde{\gamma}:(0,1)\to\widetilde{R}$ be an analytic arc and let $\widehat\gamma:(0,1)\to\widehat{R}$ be its lifting by $\sigma$. Then, $\sigma\circ\widehat\gamma=\widetilde\gamma$. We claim that 
\begin{equation}
\label{eq:Gogamma}
\widetilde{f}\circ\widetilde{\gamma} = (\widehat{f}\circ\widehat\gamma) \cdot (h \circ \pi \circ \widetilde{\gamma})\,,
\end{equation}
which implies that $\widetilde{f}\circ \widetilde{\gamma}$ is analytic. Indeed, if $\widetilde{\gamma}(t) \notin Z$, then \eqref{eq:Gogamma} holds because $(\widehat{f}\circ\sigma^{-1}\circ\widetilde\gamma)(t)=(\widehat{f}\circ\sigma^{-1}\circ\sigma\circ\widehat\gamma)(t)=(\widehat{f}\circ\widehat\gamma)(t)$. If, in turn, $\widetilde{\gamma}(t) \in Z$, then $(h \circ \pi \circ \widetilde{\gamma})(t)=0$, by definition of $h$, and hence both sides of \eqref{eq:Gogamma} are equal to zero. 

Now, we push down $\widetilde{f}$ by $\pi$ in order to get an arc analytic function on $\R^n$. More precisely, we define $f:\R^n \to \R$ as
\[
f(x) \coloneqq\begin{cases}(\widetilde{f}\circ\pi^{-1})(x), &x\notin Y\\ h^3(x), &x \in Y\,.\end{cases}
\]
To see that $f$ is arc-analytic, let $\gamma:(0,1)\to\R^n$ be an analytic arc. Let $\widetilde\gamma:(0,1)\to\widetilde{R}$ be the lifting of $\gamma$ by $\pi$, and let $\widehat\gamma:(0,1)\to\widehat{R}$ be the lifting of $\widetilde\gamma$ by $\sigma$. Then, $\pi\circ \widetilde\gamma=\gamma$, and $\sigma\circ\widehat\gamma=\widetilde\gamma$. We claim that 
\begin{equation}
\label{eq:f1ogamma}
f\circ\gamma = \widetilde{f}\circ\widetilde\gamma\,,
\end{equation}
which implies that $f \circ \gamma$ is analytic. Indeed, if $\gamma(t) \not\in Y$, then \eqref{eq:f1ogamma} holds because 
$(\widetilde{f}\circ\pi^{-1}\circ\gamma)(t)=(\widetilde{f}\circ\pi^{-1}\circ\pi\circ\widetilde\gamma)(t)=(\widetilde{f}\circ\widetilde\gamma)(t)$. If, in turn, $\gamma(t) \in Y \cap \pi(E)$, then $h(\gamma(t))=0$ and hence $(f \circ \gamma)(t)=0$. But $\widetilde{\gamma}(t) \in Z$, and hence $(\widetilde{f}\circ\widetilde\gamma)(t)=0$ as well. Finally, if $\gamma(t)\in Y\setminus\pi(E)$, then $\widetilde\gamma(t)\notin Z$ and $\widehat\gamma(t)\in\widehat{D}$; hence, by \eqref{eq:Gogamma}, we have
\begin{multline}
\notag
(\widetilde{f}\circ\widetilde\gamma)(t) = ((\widehat{f}\circ\widehat\gamma)\cdot(h\circ\pi\circ\widetilde\gamma))(t)
= \left(((h\circ\pi\circ\sigma)^2\circ\widehat\gamma)\cdot(h\circ\pi\circ\widetilde\gamma)\right)(t)\\
= \left((h\circ\pi\circ\widetilde\gamma)^2\cdot(h\circ\pi\circ\widetilde\gamma)\right)(t)
= (h\circ\pi\circ\widetilde\gamma)^3(t) = (h\circ\gamma)^3(t) = (f\circ\gamma)(t)\,.
\end{multline}
We shall now calculate the zero locus of $f$.
\begin{align*}
f^{-1}(0) & = \{x \in \R^n\setminus Y : (\widetilde{f}\circ\pi^{-1})(x)=0 \} \cup \{x \in Y : h^3(x)=0 \}\\
& = \pi\left(\{y \in \widetilde{R}\setminus D : \widetilde{f}(y)=0\}\right) \cup (X \cap Y)\\
& = \pi\left(\{y \in \widetilde{R}\setminus D : ((\widehat{f}\circ\sigma^{-1})\cdot(h\circ\pi))(y)=0\}\right) \cup (X \cap Y)\\
& = \pi\left(\{y \in \widetilde{R}\setminus D : (\widehat{f}\circ\sigma^{-1})(y)=0\}\right) \cup (X \cap Y)\\
& = \left((\pi\circ\sigma)(\{z\in\widehat{R}\setminus\sigma^{-1}(D) :  \widehat{f}(z)=0 \})\right) \cup (X \cap Y)\\
& = (\pi\circ\sigma)(\widehat{E}\setminus\sigma^{-1}(D)) \cup (X \cap Y)\\
& = \pi(E\setminus D) \cup (X \cap Y)\\
& = (\overline{\Reg_{k}(X)} \,\setminus Y) \cup (X \cap Y)\,.
\end{align*}
It follows that the $\AR$-closed set $f^{-1}(0)$ is contained in $X$ and contains the $\AR$-closure $\overline{\Reg_k(X)}^{\AR}$\!.
By $\AR$-irreducibility of $X$ however, the set $\Reg_k(X)$ is $\AR$-dense in $X$. We thus get $f^{-1}(0)=X$, which completes the proof.
\qed


\section{Algebro-geometric duality}
\label{sec:duality}

Theorem~\ref{thm:main} allows one to establish a direct dictionary between the $\AR$-closed sets and the radical ideals in the ring of arc-analytic semialgebraic functions, analogous to the case of Zariski topology over an algebraically closed field.

Let $X$ and $Y$ be $\AR$-closed sets in $\R^n$, with $Y\subset X$. We will denote by $\I_X(Y)$ the ideal in $\A(X)$ of the functions that vanish on $Y$. If $X=\R^n$, we will write $\I(Y)$ for $\I_X(Y)$. For an ideal $I$ in $\A(X)$, we will denote by $\V_X(I)$ the set of points $x\in X$ at which all  $f\in I$ vanish. Further, $\rad_X(I)$ will denote the radical of $I$ in $\A(X)$. If $X=\R^n$, we shall write $\V(I)$ for $\V_X(I)$, and $\rad(I)$ for $\rad_X(I)$.

\begin{proposition}
\label{prop:VI}
Let $X$ be an $\AR$-closed subset of $\R^n$.
\begin{itemize}
\item[(i)] If $Y\subset X$ is $\AR$-closed, then $\V_X(\I_X(Y))=Y$. 
\item[(ii)] If $I$ is an ideal in $\A(X)$, then $\I_X(\V_X(I))=\rad_X(I)$.
\end{itemize}
\end{proposition}

\begin{proof} (i) By Theorem~\ref{thm:main}, given an $\AR$-closed $Y$ in $\R^n$, there exists $f\in\A(\R^n)$ such that $Y=\V(f)$. Then, the restriction $f|_X$ is in $\I_X(Y)$, and so $\V_X(\I_X(Y))\subset\V_X(f|_X)=Y$. The inclusion $\V_X(\I_X(Y))\supset Y$ is obvious.

(ii) Given an ideal $I$ in $\A(X)$, we have $\V_X(I)=\bigcap_{f\in I}f^{-1}(0)$, hence, by noetherianity of $\AR$ topology, $\V_X(I)=f_1^{-1}(0)\cap\dots\cap f_r^{-1}(0)$, for some $f_1,\dots,f_r\in I$. Setting $g\coloneqq f_1^2+\dots+f_r^2$, we get $\V_X(I)=\V_X(g)$.

Let $f\in\I_X(\V_X(I))$ be arbitrary. Since $g^{-1}(0)\subset f^{-1}(0)$, it follows from \cite[Prop.\,6.5]{Ku} that $f\in\rad_X((g))$, and hence $f\in\rad_X(I)$. This proves that $\I_X(\V_X(I))\subset\rad_X(I)$. The opposite inclusion follows from the fact that $\I_X(\V_X(I))$ is a radical ideal which contains $I$.
\qed
\end{proof}

\begin{corollary}
\label{cor:irred}
Let $X$ be an $\AR$-closed subset of $\R^n$.
\begin{itemize}
\item[(i)] An $\AR$-closet set $Y\subset X$ is irreducible iff its ideal $\I_X(Y)$ is prime.
\item[(ii)] The zero locus $\V_X(\mathfrak{p})$ of a prime ideal $\mathfrak{p}$ in $\A(X)$ is an irreducible $\AR$-closed set.
\end{itemize}
\end{corollary}

\begin{proof}
The implication from left to right in (i) follows simply from the fact that the zero locus of every arc-analytic function is $\AR$-closed. The remaining statements of the corollary, in turn, follow immediately from Proposition~\ref{prop:VI}.
\qed
\end{proof}

Let $X$ be an $\AR$-closed set in $\R^n$. Proposition~\ref{prop:VI} together with equality $\dim_{\mathrm{g}}\!X=\dim_\mathrm{K}\!X$ from Section~\ref{sec:prelim} imply immediately that every ascending sequence of prime ideals in $\A(X)$ is of finite length, and the Krull dimension of $\A(X)$ is precisely $\dim{X}$, which gives a quick alternative proof of \cite[Prop.\,6.10]{Ku}.
Proposition~\ref{prop:VI} can be also used to establish a one-to-one correspondence between irreducible components of $X$ and the minimal primes of $\I(X)$. (Recall from Section~\ref{sec:prelim} that every $\AR$-closed set $X$ has a decomposition into finitely many $\AR$-closed sets, called its $\AR$-irreducible components, none of which can be decomposed into a union of two proper $\AR$-closed subsets.)

\begin{proposition}
\label{prop:min-primes}
Let $X, Y$ be $\AR$-closed sets in $\R^n$, with $Y\subset X$. The decomposition of \,$Y$ into $\AR$-irreducible components is given by \,$Y=\V_X(\mathfrak{p}_1)\cup\dots\cup\V_X(\mathfrak{p}_r)$, where $\mathfrak{p}_1,\dots,\mathfrak{p}_r$ are precisely the minimal prime ideals of the ring $\A(X)$ which contain $\I_X(Y)$.
\end{proposition}

\begin{proof}
Let $Y_1,\dots,Y_r$ be the $\AR$-irreducible components of $Y$. Then, $Y_i \not\subset Y_j$ for $i\neq j$. Set $\mathfrak{p}_i \coloneqq \I_X(Y_i)$, $i=1,\dots,r$. By Proposition~\ref{prop:VI}(i), we have $Y_i=\V_X(\mathfrak{p}_i)$. It follows that $Y=\V_X(\mathfrak{p}_1) \cup\dots\cup\V_X(\mathfrak{p}_r)$.

By Corollary~\ref{cor:irred}(i), each $\mathfrak{p}_i$ is a prime ideal of $\A(X)$ containing $\I_X(Y)$. Suppose that $\mathfrak{p}$ is a prime ideal of $\A(X)$ such that $\I_X(Y)\subset\mathfrak{p}\subset\mathfrak{p}_i$ for some $i$. Then $Y\supset\V_X(\mathfrak{p})\supset Y_i$. By Corollary~\ref{cor:irred}(ii), $\V_X(\mathfrak{p})$ is $\AR$-irreducible, and hence there exists $j$ such that $\V_X(\mathfrak{p})\subset Y_j$. But then $Y_i\subset Y_j$, and so $i=j$. It follows that $\V_X(\mathfrak{p})=Y_i$, and hence $\mathfrak{p}=\mathfrak{p}_i$, by Proposition~\ref{prop:VI}(ii). This proves that $\mathfrak{p}_i$ is a minimal prime of $\A(X)$ which contains $\I_X(Y)$.

It remains to see that $\mathfrak{p}_1, \dots,\mathfrak{p}_r$ are all such primes. Let $\mathfrak{p}$ be a prime in $\A(X)$ containing $\I_X(Y)$. Then, $\V_X(\mathfrak{p}) \subset Y$. By irreducibility of $\V_X(\mathfrak{p})$, it follows that $\V_X(\mathfrak{p})\subset Y_i$ for some $i$. Hence, by Proposition~\ref{prop:VI} again, $\mathfrak{p}\supset\mathfrak{p}_i$.
\qed
\end{proof}


\section{Extension of arc-analytic functions}
\label{sec:extension}

Despite some close analogies with Zariski topology over an algebraically closed field (as seen in the previous section), there are still some important open questions concerning the relationship between algebra and geometry in the $\AR$ setting. We suspect that the techniques of Section~\ref{sec:proof} can be used to show that every arc-analytic semialgebraic function on an $\AR$-closed set $X$ in $\R^n$ is, in fact, a restriction of an element of $\A(\R^n)$.

\begin{conjecture}
\label{conj:ext}
Let $X\subset\R^n$ be $\AR$-closed. Every arc-analytic semialgebraic function $f:X\to\R$ can be extended to an arc-analytic semialgebraic function on the entire $\R^n$. In other words, $\A(X)\simeq\A(\R^n)/\I(X)$ as $\R$-algebras.
\end{conjecture}

\begin{remark}
\label{rem:ext}
The extension question has recently been settled for continuous rational functions. In \cite{KN}, the authors showed that the analogue of Conjecture~\ref{conj:ext} holds for the so-called hereditarily rational functions but, in general, fails for continuous rational functions.
\end{remark}

We finish the paper with some initial results on the extension of arc-analytic functions, which give a partial justification for the conjecture. Let $X\subset\R^n$ be $\AR$-closed, and let $f\in\A(X)$. We denote by $\Sing(f)$ the locus of points $x\in \Reg_k(X)$ such that $f$ is not analytic at $x$. Recall that $\Sing(f)$ is semialgebraic, and $\dim\Sing(f)\leq\dim{X}-2$ (see, e.g., \cite{BM1}, and cf. \cite[Thm.\,5.2]{Ku}).

\begin{proposition}
\label{prop:ext}
Let $X$ be an $\AR$-closed set in $\R^n$, and let $f\in\A(X)$. There exists \,$h\in\A(\R^n)$, with $\dim h^{-1}(0)<\dim{X}$, such that $fh$ can be extended to an arc-analytic semialgebraic function $F\in\A(\R^n)$.
\end{proposition}

\begin{proof}
As in the proof of Theorem~\ref{thm:main}, let $\pi:\widetilde{R}\to\R^n$ be an embedded desingularization of $\overline{X}^{\Zar}$\!, and let $E$ be the union of connected components of the strict transform $\widetilde{X}$ of $\overline{X}^{\Zar}$ such that $\pi(E)=\overline{\Reg_k(X)}$, where $k=\dim X$. By \cite[Thm.\,1.1]{BM1}, there exists a finite composition of blowings-up $\sigma:\widehat{R}\to\widetilde{R}$ (with smooth algebraic centres) which converts the arc-analytic semialgebraic function $f\circ\pi$ into a Nash function $f\circ\pi\circ\sigma$. In particular, $f\circ \pi\circ\sigma|_{\widehat{E}}$ is Nash, where $\widehat{E}$ is the strict transform of $E$ by $\sigma$. Let $Z$ be the centre of $\pi \circ \sigma$ (i.e., the subset of $\R^n$ outside of whose preimage $\pi\circ\sigma$ is an isomorphism). Notice that $Z$ is the union of the singular locus of $\overline{X}^\Zar$ and the images of all the centres of blowings-up involved in $\sigma$. Then, $Z$ is $\AR$-closed and of dimension less than $k$. 

By Theorem~\ref{thm:main}, there exists $h\in\A(\R^n)$ such that $Z\cup\overline{\Reg_{<k}(X)}^\AR\!= h^{-1}(0)$, where $\Reg_{<k}(X)$ denotes the locus of smooth points of $X$ of dimensions less than $k$. It follows that $\dim h^{-1}(0)<k$.

We can extend $f\circ\pi\circ\sigma|_{\widehat{E}}$, by \cite[Cor.\,8.9.13]{BCR}, to a Nash function $g:\widehat{R}\to\R$. Now, define $F:\R^n\to\R$ as
\[
F(x) \coloneqq \begin{cases} \left((g\circ\sigma^{-1}\circ\pi^{-1}) \cdot h \right)(x)\,,  &x\not\in Z\\ 0\,, &x\in Z\,.\end{cases}
\]
As in the proof of Theorem~\ref{thm:main}, one easily verifies that $F$ is arc-analytic on $\R^n$. By construction, $F$ is semialgebraic and satisfies $F(x)=(fh)(x)$ for all $x\in X$.
\qed
\end{proof}
\medskip

\begin{proposition}
\label{prop:ext2}
Let $X \subseteq \R^n$ be an algebraic set of pure dimension $k$, with an isolated singularity at $p\in\R^n$. Let $f\in\A(X)$ be analytic except perhaps at $p$, and suppose that $f(p)=0$. Then, there exists \,$d\geq1$ such that $f^d$ can be extended to an arc-analytic semialgebraic function $F\in\A(\R^n)$.
\end{proposition}

\begin{proof}
Let $\pi:\widetilde{R}\to\R^n$, $E$, $\sigma:\widehat{R}\to\widetilde{R}$, $Z$, and $\widehat{E}$ be as in the proof of Proposition~\ref{prop:ext}. By \cite[Thm.\,1.3]{KuPa}, the (smooth Nash) centres of the blowings-up in $\sigma$ can be chosen such that $\sigma$ is an isomorphism outside the preimage of $\Sing(f\circ\pi)$. Consequently, one can assume that $\pi\circ\sigma$ is an isomorphism outside the preimage of $p$ (i.e., $Z=\{p\}$). By assumptions on $X$, it also follows that $\widetilde{E}$ (resp. \!$\widehat{E}$) is the entire strict transform of $X$ by $\pi$ (resp. \!{by} $\pi\circ\sigma$).

Let $h(x):=\|x-p\|^2$ be the square of the Euclidean distance from the point $p$. Then, by \cite[Prop.\,6.5]{Ku}, there exists $d\geq1$ and a function $f_1\in\A(X)$ such that $f^d=f_1\!\cdot\!h$.
Since $h$ is analytic on $\R^n$ and non-vanishing outside of $p$, it follows that $\Sing(f_1)\subset\Sing(f)$. One can thus, without loss of generality, assume that $\sigma$ converts the arc-analytic semialgebraic function $f_1\circ\pi$ into a Nash function $f_1\circ\pi\circ\sigma$.
As in the proof of Proposition~\ref{prop:ext}, let $g\in\Na(\widehat{R})$ be an extension of the function $f_1\circ\pi\circ\sigma|_{\widehat{E}}$. Define
\[
F(x) \coloneqq \begin{cases} \left((g\circ\sigma^{-1}\circ\pi^{-1}) \cdot h \right)(x)\,,  &x\neq p\\ 0\,, &x=p\,.\end{cases}
\]
As in the proof of Theorem~\ref{thm:main}, one easily verifies that $F$ is arc-analytic on $\R^n$. By construction, $F$ is semialgebraic and satisfies $F(x)=f^d(x)$ for all $x\in X$.
\qed
\end{proof}


\end{document}